\documentclass[12pt]{article}

\usepackage{amssymb,amsmath,amsthm}
\usepackage{mathrsfs}
\usepackage{geometry}
\usepackage{tikz-cd}
\usepackage[shortlabels]{enumitem}
\usepackage{mathtools}
\usepackage[titletoc,title]{appendix}

\theoremstyle{plain}
\newtheorem*{theorem*}{Theorem}
\newtheorem{theorem}{Theorem}[section]
\newtheorem{proposition}[theorem]{Proposition}
\newtheorem{lemma}[theorem]{Lemma}

\newtheorem*{conjecture*}{Conjecture}

\theoremstyle{definition}
\newtheorem{definition}[theorem]{Definition}

\DeclareMathOperator{\codim}{codim}
\DeclareMathOperator{\rank}{rank}
\DeclareMathOperator{\Proj}{Proj}
\DeclareMathOperator{\Spec}{Spec}
\DeclareMathOperator{\Sym}{Sym}

\DeclareMathOperator{\blow}{Bl}
\DeclareMathOperator{\id}{id}

\newcommand{\ZZ}{\mathbb{Z}}

\newcommand{\PP}{\mathbb{P}}

\newcommand{\EEs}{\mathscr{E}}
\newcommand{\FFs}{\mathscr{F}}

\newcommand{\GGs}{\mathscr{G}}
\newcommand{\OOs}{\mathscr{O}}
\newcommand{\LLs}{\mathscr{L}}

\begin{document}

\title{A relative Segre zeta function}
\author{Grayson Jorgenson}
\date{}
\maketitle

\begin{abstract}
The choice of a homogeneous ideal in a polynomial ring defines a closed subscheme $Z$ in a projective space as well as an infinite sequence of cones over $Z$ in progressively higher dimension projective spaces. Recent work of Aluffi introduces the Segre zeta function, a rational power series with integer coefficients which captures the relationship between the Segre class of $Z$ and those of its cones. The goal of this note is to define a relative version of this construction for closed subschemes of projective bundles over a smooth variety. If $Z$ is a closed subscheme of such a projective bundle $P(E)$, this relative Segre zeta function will be a rational power series which describes the Segre class of the cone over $Z$ in every projective bundle ``dominating'' $P(E)$. When the base variety is a point we recover the absolute Segre zeta function for projective spaces. Part of our construction requires $Z$ to be the zero scheme of a section of a bundle on $P(E)$ of rank smaller than that of $E$ that is able to extend to larger projective bundles. The question of what bundles may extend in this sense seems independently interesting and we discuss some related results, showing that at a minimum one can always count on direct sums of line bundles to extend. Furthermore, the relative Segre zeta function depends only on the Segre class of $Z$ and the total Chern class of the bundle defining $Z$, and the basic forms of the numerator and denominator can be described. As an application of our work we derive a Segre zeta function for products of projective spaces and prove its key properties.
\end{abstract}

\section{Introduction}

\subsection{Motivation}

For a homogeneous ideal $I\subseteq k[x_0,\ldots,x_n]$ where $k$ is any field, one can consider the Segre class $s(Z,\PP^n)$ of the closed subscheme $i: Z := \Proj(k[x_0,\ldots,x_n]/I) \hookrightarrow \PP^n$ defined by $I$. This class lives in the Chow group $A_*(Z)$ of $Z$ and is a central object in Fulton-MacPherson intersection theory \cite{fulton1}, containing information about the embedding of $Z$ into $\PP^n$. As it is usually the case that $A_*(Z)$ lacks a simple presentation, one often considers the pushforward $i_*s(Z,\PP^n)\in A_*(\PP^n)\cong \ZZ[t]/(t^{n+1}).$ Key information from the Segre class carries over to its pushforward and may be extracted from the coefficients of the minimum degree integer polynomial representing $i_*s(Z,\PP^n)$. Consequently the problem of computing this polynomial has received a substantial amount of attention in the past several years \cite{aluffi4} \cite{eklund1} \cite{harris1} \cite{helmer1} resulting in a number of algorithms to compute the Segre class of such a $Z$, including one currently implemented in the \texttt{CharacteristicClasses} package of the Macaulay2 computer algebra system \cite{grayson_stillman1}.

For $N\geq n$, if we denote by $I_N$ the extension of the ideal $I$ to the larger polynomial ring $k[x_0,\ldots,x_N]$ via the inclusion, we may consider the closed subscheme $i_N: Z_N\hookrightarrow \PP^N$ cut out by $I_N$. We call this the \emph{cone} over $Z$ in $\PP^N$ since $\PP^n$ embeds into $\PP^N$ as the linear subvariety cut out by $x_{n+1},\ldots,x_N$, and on closed points $Z$ is the image of $Z_N$ by the linear projection of $\PP^N$ onto $\PP^n$ away from the subvariety defined by $x_0,\ldots,x_n$. It is natural to ask whether the Segre classes of the members of this infinite sequence of cones are all related.

The answer was recently given by Aluffi \cite{aluffi1} in the \emph{Segre zeta function} $\zeta_I(t)$ of $I$, which is defined to be the power series $$\zeta_I(t) = \sum_{j\geq 0}a_jt^j\in \ZZ[[t]]$$ satisfying $$\zeta_I(H_N) = (i_N)_* s(Z_N,\PP^N)$$ for every $N\geq n$, where $H_N\in A_*(\PP^N)$ is the hyperplane class.

This is well-defined and the series turns out to be rational \cite[Theorem 5.8]{aluffi1}. Rationality implies that in order to know the pushforward of the Segre class of each of the cones in this infinite sequence it suffices to know just one of the $(i_N)_* s(Z_N,\PP^N)$ for a large enough $N$. When one picks a homogeneous generating set for $I$ the rational expression equal to $\zeta_I(t)$ and the $N$ required above can be described more explicitly. These facts, when used in conjunction with existing Segre class algorithms, can be used to significantly improve the speed of computing $i_*s(Z,\PP^n)$ in some cases \cite[Example 5.11]{aluffi1}. The Segre zeta function itself is a direct source of other invariants of projective schemes as well; for example, if $Z$ is a smooth variety, the polar degrees of $Z$ appear in the coefficients of $\zeta_I(-\frac{t}{1+t})$ \cite[Proposition 6.2]{aluffi1}.

The purpose of this note is to explore a generalization of the Segre zeta function. The linear projection mentioned earlier is the rational map $\PP^N\dashrightarrow \PP^n$ corresponding to the inclusion $k[x_0,\ldots,x_n]\hookrightarrow k[x_0,\ldots,x_N]$. Alternatively, one may derive this rational morphism from the natural surjection of bundles $O^{\oplus(N+1)}_{\{\mathrm{pt}\}}\rightarrow O^{\oplus(n+1)}_{\{\mathrm{pt}\}}$ over a point where $\PP^N, \PP^n$ are interpreted as the projective bundles associated to these vector bundles.

Therefore, more generally, one could try to define cones in all projective bundles over a smooth variety that dominate a given one, where we say one projective bundle dominates another if there is a surjection of the underlying vector bundles. In this case that surjection of bundles induces a rational map of projective bundles that is fiberwise a linear projection. This is the nature of the generalization we establish here; we will construct a \emph{relative} Segre zeta function that describes the Segre classes of the cones in these larger projective bundles.

Note that in the absolute case, only projective spaces are considered, and these organize into an infinite \emph{sequence} of projections from progressively larger projective spaces.

\begin{center}
\begin{tikzcd}
Z \arrow[d, hookrightarrow] & {} & {} & {} & {} & {} & {}\\
\PP^n & \PP^{n+1} \arrow[l, dashrightarrow] & \PP^{n+2} \arrow[l, dashrightarrow] & \PP^{n+3} \arrow[l, dashrightarrow] & \PP^{n+4} \arrow[l, dashrightarrow] & \PP^{n+5} \arrow[l, dashrightarrow] & \cdots \arrow[l, dashrightarrow]
\end{tikzcd}
\end{center}

However, an essential feature of our more general context is that there can be many non-isomorphic projective bundles of the same dimension dominating any one projective bundle $P(E)$ \cite[Theorem 9.5]{eisenbud1}, and these make up an infinite \emph{tree} rooted at $P(E)$, with possibly many branches. When we have a sufficiently nice closed subscheme $Z$ of $P(E)$ our relative Segre zeta function will describe the Segre class of the cones over $Z$ in all these larger projective bundles.

More specifically suppose we have such a rational map $P(F)\dashrightarrow P(E)$ of projective bundles over a smooth variety $X$, and suppose $i:Z\hookrightarrow P(E)$ is a closed subscheme of $P(E)$ which is the zero scheme of a section $s$ of a bundle $G$ on $P(E)$. Then provided the rank of $G$ is less than that of $E$ and provided there exists a bundle on $P(F)$ and a section of this bundle sufficiently ``compatible'' with $G$ and $s$, a notion which we will make precise below, we may define a subscheme $\hat{Z}$ of $P(F)$ which will be the cone over $Z$ in a sense that generalizes the cones of the absolute case. The problem of finding compatible bundles and sections seems independently interesting and is related to the difficulties with extending bundles from open subschemes. At least when $G$ splits as a sum of line bundles, such an extension may always be found.

In this situation we define the relative Segre zeta function to be a formal power series $\zeta_{G,s}(t)$ with coefficients from $A_*(X)$, with the property that when we evaluate it at the tautological class $c_1(O_{P(E)}(1)\cap [P(E)]$ of $A_*(P(E))$ we obtain the pushforward of $s(Z,P(E))$ to $A_*(P(E))$, and we obtain the pushforward of $s(\hat{Z},P(F))$ when we evaluate $\zeta_{G,s}(t)$ at the tautological class of $A_*(P(F))$. Furthermore this series is rational by construction, $$\zeta_{G,s}(t) = \frac{P(t)}{Q(t)},$$ for some polynomials $P(t), Q(t)\in A_*(X)[t]$ depending only on the representations of the classes $c(G)\cap i_*s(Z,P(E))$ and $c(G)\cap [P(E)]$ in $A_*(P(E))$.

As an immediate consequence of our work we derive a Segre zeta function for products of projective spaces which has already been used in \cite{aluffi2} as a device for expressing the Chern-Schwartz-MacPherson class of a smooth projective variety \cite[\S 4]{aluffi2}. In the following part of this section, we introduce needed notation and state our results precisely.

\subsection{Statement of the results}

Throughout we take our base field $k$ to be algebraically closed and for us varieties are integral schemes separated and of finite type over $k$. Let $E$ be a rank $e$ algebraic vector bundle on a smooth, $n$-dimensional variety $X$. We denote by $P(E)$ the projective bundle of lines in $E$. The analogy to choosing a homogeneous ideal in the projective space setting to specify a closed subscheme is choosing a bundle $G$ and one of its sections and indeed every closed subscheme of a projective bundle over a smooth variety is the zero scheme of such a section \cite[B.8.1]{fulton1} 

If $G$ is a bundle on $P(E)$, $s: P(E)\rightarrow G$ any section of $G$, and $s_0$ the zero section embedding of $P(E)$ into $G$, then the zero scheme of $s$ is the scheme $Z$ completing the following fiber square:
\begin{center}
\begin{equation}
\begin{tikzcd}
Z := P(E) \times_{G} P(E) \arrow[r, hookrightarrow] \arrow[d, hookrightarrow, "i"] & P(E) \arrow[d, hookrightarrow, "s"]\\
P(E) \arrow[r, hookrightarrow, "s_0"] & G
\end{tikzcd}
\end{equation}
\end{center} % i should indeed be the embedding on the left. Think set-theoretically. We want i to be the embedding that represents the inclusion of Z into P(E). It is for this map that $G$ is the normal bundle, see chap. 6 Fulton. Note: by corollary 6.5 of Fulton, all sections of a vector bundle over X are regular embeddings of X into that vector bundle

Just like with the Chow rings of projective spaces, the Chow ring of $P(E)$ has a nice description as a ring: $$A_*(P(E))\cong A_*(X)[t]/(t^e + c_1(E)t^{e-1} + \ldots + c_{e}(E))$$ where the isomorphism identifies $c_1(O_{P(E)}(1))\cap [P(E)]$ with $t$. For simplicity of notation, whenever we invoke this identification we will abbreviate the Chern class notation; in this isomorphism $c_i(E)$ refers to $c_i(E)\cap[X]$. By pushing forward the Segre class $s(Z,P(E))$ of $Z$ in $P(E)$ to $P(E)$ via the embedding $i$ and using the above isomorphism, we obtain a class we can describe as a polynomial in $t$ with coefficients from $A_*(X)$. Tautologically we may write $$i_* s(Z,P(E)) = c(G)^{-1}\cap(c(G)\cap i_*s(Z,P(E))),$$ which presents $i_* s(Z,P(E))$ as a sort of rational expression in $t$.

\begin{definition}
More precisely, let $P(t)\in A_*(X)[t]$ be the lowest degree polynomial representing the class $$c(G)\cap i_*s(Z,P(E))\in A_*(X)[t]/(t^e + c_1(E)t^{e-1} + \ldots + c_{e}(E)),$$ and $Q(t)$ the lowest degree polynomial representing the class $c(G)\cap [P(E)]$. We define the \emph{Segre zeta function} of $G$ with respect to $s$ to be the rational expression $$\zeta_{G,s}(t) := \frac{P(t)}{Q(t)}\in A_*(X)[[t]].$$
\end{definition}

%The key idea behind this is that when the rank of $G$ is sufficiently small, this ``same'' rational expression will describe the Segre classes of certain subschemes of larger projective bundles which will serve as our generalizations of the notion of a cone in the projective space case.

Suppose we are given a surjection of bundles $F\rightarrow E$. Such a surjection induces a rational map of the corresponding projective bundles $\phi: P(F)\dashrightarrow P(E)$ and we let $U$ be the open subscheme of $P(F)$ where $\phi$ is a morphism. To define what we mean by a cone over $Z$ in $P(F)$, we try to find a bundle on $P(F)$ and section of said bundle which will serve as a kind of extension of $G$ and $s$ to $P(F)$.  Specifically, we require a bundle $\hat{G}$ on $P(F)$ and a section $\hat{s}$ of $\hat{G}$ satisfying the following two conditions:

\begin{enumerate}[label=\roman*]
\item there exists a bundle isomorphism $\hat{G}\big|_U\cong (\phi\big|_U)^*G$,
\item there is a section $\hat{s}: P(F)\rightarrow \hat{G}$ making
\begin{center}
\begin{tikzcd}
\hat{G}\big|_U \arrow[r, leftrightarrow, "\cong"] & (\phi\big|_U)^* G \arrow[dl, leftarrow, "(\phi\big|_U)^* s"]\\
U \arrow[u, "\hat{s}\big|_U"] & {}
\end{tikzcd}
\end{center}
commute.
\end{enumerate}

If such a bundle and section exist we define the \emph{cone} in $P(F)$ over $Z$, denoted $\hat{Z}$, to be the zero scheme of $\hat{s}$. This scheme can still be thought of as a cone in a geometric sense; indeed the rational map $\phi$ is fiberwise a projection from a linear subvariety, and on closed points, $$\hat{Z} = \overline{(\phi\big|_U)^{-1}(Z)}.$$ Note that even when $Z$ is reduced it does not suffice in general to define $\hat{Z}$ by taking the reduced scheme structure on the set-theoretic closure of $(\phi\big|_U)^{-1}(Z)$: in the absolute case for instance, the closed subscheme of $\PP^1(x:y)$ defined by the ideal $(x^2, xy)$ is reduced but the cone cut out by the same equations in $\PP^2(x:y:z)$ is not. % the saturation of this ideal in P^1 is (x), so is reduced (also could check on both affine charts: the coordinate ring is reduced in both cases), and in P^2 it is nonreduced (localizing at z, one obtains the ideal (x^2, xy) in A^2(x,y). But the coordinate ring k[x,y]/(x^2,xy) has [x] as a nilpotent element).

Our main result is:

\begin{theorem}
Suppose $g < e$, let $F\rightarrow E$ be any surjection of bundles on $X$, and let $\phi: P(F)\dashrightarrow P(E)$ be the induced rational map. If $\hat{G}$ is a bundle on $P(F)$ with section $\hat{s}$ which satisfy properties (i), (ii) above, then letting $\hat{i}: \hat{Z}\hookrightarrow P(F)$ denote the zero scheme of $\hat{s}$, we have $$\hat{i}_* s(\hat{Z},P(F)) = \zeta_{G, s}(c_1(O_{P(F)}(1))\cap[P(F)]).$$
\end{theorem}

In other words, when the criteria of Theorem 1.2 are satisfied, the ``same'' rational expression that gives us $i_*s(Z,P(E))$ also gives $\hat{i}_*s(\hat{Z},P(F))$. The rank constraint is necessary to ensure that the representations for $c(G)\cap i_*s(Z,P(E))$ and $c(G)\cap [P(E)]$ are unaffected by the relation defining the presentation of $A_*(P(E))$.

We in fact know a little more about the numerator of $\zeta_{G,s}(t)$. The highest dimension term of $c(G)\cap i_*s(Z,P(E))$ is the top dimension term of $i_*[Z]$, where $[Z]$ is the fundamental class of $Z$ in $A_*(Z)$, see \cite[\S 1.5, Example 4.3.4]{fulton1}. We will show in Section 2 that the lowest dimension term of $c(G)\cap i_*s(Z,P(E))$ is $c_g(G)\cap [P(E)]$. We collect these facts into an auxiliary proposition.

\begin{proposition}
For $a\in A^j(X)$, we call $j+d$ the \emph{total degree} of the monomial $at^d\in A_*(X)[t]$. We have:
\begin{enumerate}[(a)]
\item the term of $P(t)$ of lowest total degree is the minimal degree polynomial in $A_*(X)[t]$ representing the top dimension part of $i_*[Z]$, which is of total degree $\codim(Z)$,
\item the term of $P(t)$ of highest total degree is the minimal degree polynomial representing $c_g(G)\cap [P(E)]$, which is of total degree $g$, and this is the same as the term of highest total degree of the denominator $Q(t)$.
\end{enumerate}
\end{proposition}% the reason this last part is true is as follows: we simply know that the highest dimension part of $A_*(Z)$ is generated as a free Z-module over the top dimension irreducible components of $Z$. Furthermore, Segre classes have the property that the coefficient of [V] in the expression for s(Z,\PP^n\times\PP^m) for an irreducible component V of Z is the multiplicity of $\PP^n\times\PP^m$ along Z at $V$. Since the ambient scheme (\PP^n\times\PP^m) is smooth, this multiplicity should just be the multiplicity of $Z$ along $V$. As [Z] is the sum of the classes of the irreducible components of $Z$ with coefficients the multiplicity of Z along each irreducible components, this means that in particular the top dimension part of s(Z,\PP^n\times\PP^m) is the top dimension part of [Z].

By using facts about the extension of reflexive sheaves, we note that $\hat{G}$ is unique up to isomorphism, and show that such a bundle and section always exist at least when $G$ splits as a sum of line bundles. Additionally, every closed subscheme of $P(E)$ can be expressed as the zero scheme of a section of a bundle that splits into a sum of line bundles.

\begin{proposition}
Consider a surjection of bundles on $X$, $F\rightarrow E$ and let $G$ be any bundle on $P(E)$ with section $s$. Suppose that the rank $e$ of $E$ is $\geq 2$.
\begin{enumerate}[(a)]
\item Any bundle $\hat{G}$ with section $\hat{s}$ satisfying conditions (i) and (ii) of Theorem 1.2 is unique up to isomorphism.
\item If $G$ splits as a sum of line bundles then such a $\hat{G}$ and $\hat{s}$ exist.
\item Every closed subscheme $Z$ of $P(E)$ is the zero scheme of a section of a direct sum of line bundles.
\item If $Z$ is the zero scheme of a section of a bundle as in (b), (c), and if the rank of that bundle is $< e$, then Theorem 1.2 may always be applied.
\end{enumerate}
\end{proposition}

So for instance, if one is able to choose fewer than $e$ effective Cartier divisors on $P(E)$ whose scheme-theoretic intersection is $Z$, then this is a situation where we may always successfully apply Theorem 1.2, taking $G$ to be the sum of line bundles corresponding to the chosen divisors.

The case where $e = 1$ is trivial in the sense that the Segre zeta function becomes unnecessary. When $e = 1$, $P(E) = X$, and any rational map $P(F)\dashrightarrow P(E)$ commuting with the projections to $X$ extends to the projection map $q: P(F)\rightarrow X$. So if we had a closed subscheme $Z$ of $P(E) = X$, the subscheme $q^{-1}(Z)$ of $P(F)$ would be the canonical choice for the cone over $Z$ in $P(F)$, and we can use $q$ to directly relate our Segre classes: $q^*s(Z, P(E)) = s(q^{-1}(Z), P(F))$ \cite[Proposition 4.2 (b)]{fulton1}.

Finally, we may reduce the ambient space dimension of the projective bundle we are working with if given the existence of a sufficiently transverse codimension one subbundle of $E$.

\begin{proposition}
Suppose $l: E^\prime\hookrightarrow E$ is a subbundle of rank $e^\prime = e - 1$, and suppose $g < e^\prime$. Then $l$ induces an embedding $l: P(E^\prime)\hookrightarrow P(E)$, and $l^* G$ and $l^*s$ are a bundle and section pair on $P(E^\prime)$ which cut out a subscheme $i^\prime: Z^\prime\hookrightarrow P(E^\prime)$. If $P(E^\prime)$ does not contain any of the supports of the irreducible components of the normal cone of $Z$ in $P(E)$, then in this case $$\zeta_{G,s}(t) = \zeta_{l^*G, l^*s}(t).$$
\end{proposition}

This note is organized as follows. Sections 2 and 3 are dedicated to the proofs of Theorem 1.2, Proposition 1.4, respectively, and Proposition 1.3 will be also be proven in Section 2. In Section 4 we discuss the special situation where the surjection of bundles we are working with splits, $F = E\oplus L$, for some line bundle $L$ on $X$, in which case Theorem 1.2 has a more basic proof. There we also prove Proposition 1.5. In Section 5 we show how our generalization specializes to the Segre zeta function of the absolute case, and as an application of the relative Segre zeta function we derive a Segre zeta function for products of projective spaces.

\textbf{Acknowledgements.} I wish to thank Paolo Aluffi for suggesting this problem, and for his guidance and support.

\section{Proof of the main result}

As in Section 1.2, let $\phi : P(F)\dashrightarrow P(E)$ be a rational map of projective bundles on a smooth variety $X$, induced by a surjection of the underlying bundles. Let $U$ be the open subscheme of $P(F)$ where $\phi$ is a morphism, and let $e = \rank(E), f = \rank(F)$. Following Aluffi \cite{aluffi1}, we say that a bundle $\hat{G}$ on $P(F)$ is \emph{compatible} with a bundle $G$ on $P(E)$ if condition (i) of Section 1.2 is satisfied, that is, if there is an isomorphism $\hat{G}\big|_U \cong (\phi\big|_U)^* G$. Likewise, if $s$ is a section of $G$, then a section $\hat{s}$ of $\hat{G}$ is \emph{compatible} to $s$ if condition (ii) is satisfied. We organize the proof of Theorem 1.2 as a sequence of lemmas.

\begin{lemma}
The hyperplane bundle on $P(F)$, $O_{P(F)}(1)$, is compatible with that of $P(E)$.
\end{lemma}
\begin{proof}
Note that if $\EEs, \FFs$ are the locally free sheaves of sections of $E$, $F$, respectively, then we have $E = \Spec(\Sym(\EEs^\vee))$, $F = \Spec(\Sym(\FFs^\vee))$, and $P(E) = \Proj(\Sym(\EEs^\vee)), P(F) = \Proj(\Sym(\FFs^\vee))$. We can find a cover of $X$ with affine open subschemes small enough so that the restrictions of $\EEs, \FFs$ to each are simultaneously free. Let $W = \Spec(A)$ be one such subscheme.

Then $$P(E)_W := P(E)\big|_W = \Proj(A[x_1,\ldots,x_e]), P(F)_W := P(E)\big|_W = \Proj(A[y_1,\ldots,y_f]).$$ The surjection $F_W \rightarrow E_W$ corresponds to a degree-preserving, injective homomorphism of graded rings, $A[x_1,\ldots,x_e]\hookrightarrow A[y_1,\ldots,y_f]$. This in turn defines the rational map $$\phi\big|_{U\cap P(F)_W}: P(F)_W\dashrightarrow P(E)_W.$$

In this situation, the restrictions of the sheaves of sections of the hyperplane bundles are isomorphic over $U$, that is, $$(\phi\big|_{U\cap P(F)_W})^* \OOs_{P(E)}(1)\big|_{P(E)_W} \cong \OOs_{P(F)}(1)\big|_{U\cap P(F)_W}$$ \cite[Proposition II.5.12 (c)]{hartshorne1}. These local isomorphisms are compatible with restriction and thus glue together to yield an isomorphism globally: $$O_{P(F)}(1)\big|_U \cong (\phi\big|_U)^* O_{P(E)}(1).$$

\end{proof}

Let $K$ be the kernel of $F\rightarrow E$ which is thus a subbundle of $F$. The indeterminancy locus of $\phi$ is then $P(K)\hookrightarrow P(E)$, and we may resolve $\phi$ by blowing up $P(F)$ along $P(K)$. This produces the commutative diagram:

\begin{center}
\begin{tikzcd}
{} & \blow_{P(K)}P(F) \arrow[dr, "\lambda"] & {}\\
P(F) \arrow[ur, leftarrow, "\pi"] \arrow[rr, dashrightarrow, "\phi"] & {} & P(E) \arrow[dl, "p"]\\
{} & X \arrow[ul,  leftarrow, "q"] & {}
\end{tikzcd}
\end{center} % note we always have such a diagram: the blow up can be seen as the closure of the graph of the rational map in the product space $P(F)\times P(E)$. Then $\pi$ will always be birational and proper. The key thing is that the other map, the one to $P(E)$, is flat. (see 9.3.2 of Eisenbud&Harris 3264. There they actually compute the bundle to which this projective bundle corresponds). Projection maps from projective bundles are in particular flat. 

% An argument: by the universal property of blow-ups, Bl_{q^-1(a)\cap P(K)}q^-1(a) embeds into Bl_{P(K)}P(E) naturally as the proper transform, \overline{\pi^{-1}(q^-1(a)\setminus P(K))}. By Eisenbud&Harris, the smaller blowup is a projective bundle over p^-1(a). Playing with the large diagram for the situation shows that the fibers of points of P(E) over a in the blowup are the fibers from the smaller blowup over the points of p^-1(a), which are projective spaces of dimension (f -1) - (e-1).

%ACTUALLY: if can just show fibers are all constant of the correct dimension, then this will suffice to show map is flat (see "miracle flatness"). So maybe can do this by using this 3264 result on the fibers.

% actually, even if $\lambda$ were not flat, it is a morphism of smooth varieties, so a pullback map is still defined. It has all the needed functoriality properties. Does Aluffi's proof of Lemma 2.6 still go through in this case?

Here $\pi$ is a proper birational map, $\lambda$ is the induced map to $P(E)$, and $p,q$ are the projections from $P(E), P(F)$ to $X$, respectively. The map $\lambda$ realizes the blow up as a projective bundle over $P(E)$ (see \cite[Proposition 9.3.2]{eisenbud1} for the absolute case) and so it is in particular a flat morphism. This gives us a way to move a class from $A_*(P(E))$ to $A_*(P(F))$. Following Aluffi \cite{aluffi1} once more, we define for a class $\alpha\in A_*(P(E))$ the \emph{join} of $\alpha$ with $P(K)$ to be $$\alpha\vee P(K) := \pi_*\lambda^*\alpha\in A_*(P(F)).$$

\begin{lemma}
Let $d < e$. Then $$(c_1(O_{P(E)}(1))^d\cap [P(E)])\vee P(K) = c_1(O_{P(F)}(1))^d\cap [P(F)].$$
\end{lemma}
\begin{proof}
Let $V$ denote the complement of the exceptional divisor in $\blow_{P(K)}P(F)$, so $V = \pi^{-1}(U)$. By the previous lemma, the $O(1)$ bundles on $P(F), P(E)$ are compatible, and thus we obtain an isomorphism of bundles on the blow-up $\blow_{P(K)}P(F)$, $$(\lambda^*O_{P(E)}(1))\big|_{V} \cong (\pi\big|_{V})^*(O_{P(F)}(1)\big|_U).$$ In particular, this means these two bundles have the same first Chern class. By the commutativity of the blow-up diagram and by the functorial properties of Chern classes, if $j: V\hookrightarrow \blow_{P(K)}P(F)$ denotes the open immersion, we have $$j^*(c_1(\pi^*O_{P(F)}(1))^d\cap [\blow_{P(K)}P(F)] - \lambda^*(c_1(O_{P(E)}(1))^d\cap [P(E)])) = 0.$$

From Fulton \cite[Proposition 1.8]{fulton1}, the sequence
\begin{center}
\begin{tikzcd}
A_*(D) \arrow[r, "l_*"] & A_*(\blow_{P(K)}P(F)) \arrow[r, "j^*"] & A_*(V) \arrow[r] & 0
\end{tikzcd}
\end{center}
is exact, where $l: D\hookrightarrow \blow_{P(K)}P(F)$ denotes the exceptional divisor. So knowing that $$c_1(\pi^*O_{P(F)}(1))^d\cap [\blow_{P(K)}P(F)] - \lambda^*(c_1(O_{P(E)}(1))^d\cap [P(E)]) \in \ker j^*$$ tells us this class is equal to the pushforward by $l$ of a class $\alpha\in A_{n + f - 1 - d}(D)$.

Thus as $\dim P(K) = n + f - e - 1$, so long as $d < e$, we see that $\pi_*l_* \alpha = 0$. Therefore, we may conclude by the projection formula for Chern classes \cite[Theorem 3.2 (c)]{fulton1}.
\end{proof}

A completely analogous argument establishes:

\begin{lemma}
If $G$ is a rank $g < e$ bundle on $P(E)$, $\hat{G}$ a compatible bundle to $G$ on $P(F)$, then $(c(G)\cap[P(E)])\vee P(K) = c(\hat{G})\cap [P(F)]$.
\end{lemma}

The isomorphism $A_*(P(E))\cong A_*(X)[t]/(t^e + c_1(E)t^{e - 1} + \ldots + c_e(E))$ is defined so that the expression $\alpha t^d$ corresponds to $c_1(O_{P(E)}(1))^d\cap p^*\alpha$, for any $\alpha\in A_*(X)$. Likewise for $A_*(P(F))$. This, the previous lemma, the commutativity of the blow-up diagram, and the projection formula for Chern classes together gives us the following:

\begin{lemma}
If $d < e$, then for any $\alpha_j\in A_*(X)$, $(\alpha_d t^d + \ldots + \alpha_0)\vee P(K) = \alpha_d t^d + \ldots + \alpha_0$.
\end{lemma}

Next, assuming we have a rank $g$ bundle $G$ on $P(E)$, a section $s$ of $G$, and by setting $i: Z\hookrightarrow P(E)$ to be the zero scheme of $s$, we show:

\begin{lemma}
The class $$c(G)\cap i_*s(Z,P(E))\in A_*(X)[t]/(t^e + c_1(E)t^{e-1} + \ldots + c_{e}(E))$$ can be represented by a polynomial in $A_*(X)[t]$ of degree $\leq g$.
\end{lemma}
\begin{proof}
The key point to show here is that the class $c(G)\cap i_*s(Z,P(E))$ has no homogeneous components of codimension $> g$. This is Fulton \cite[Example 6.1.6]{fulton1}, as the diagram (1) is the setup for intersecting the image of $P(E)$ by $s$ with the zero section embedding of $P(E)$ in $G$; the zero section embedding is a regular embedding with normal bundle $G$. Let $N = i^* G$.

To see why this is the case in more detail, note that by Fulton \cite[B.5.7]{fulton1}, there is an exact sequence $$0 \rightarrow O_{P(N\oplus 1)}(-1) \rightarrow r^*(N\oplus 1) \rightarrow \xi \rightarrow 0$$ of bundles on $P(N\oplus 1)$, where $r: P(N\oplus 1)\rightarrow Z$ is the projection map, and $\xi$ is the universal rank $g$ quotient bundle on $P(N\oplus 1)$.

The normal cone of $Z$ in $P(E)$, $C:= C_Z P(E)$ \cite[Chapter 4]{fulton1} embeds into $N$, and so we may view its projective completion $P(C\oplus 1)$ as a closed subscheme of $P(N\oplus 1)$. Then, by the sum formula for Chern classes we have $$r_*(c(\xi)\cap [P(C\oplus 1)]) = r_*(c(r^*(N\oplus 1))s(O_{P(N\oplus 1)}(-1))\cap [P(C\oplus 1)]).$$

The expression $$c(r^*(N\oplus 1))s(O_{P(N\oplus 1)}(-1))\cap [P(C\oplus 1)]$$ is equal to $$c(r^*(N\oplus 1))\cap (\sum_{j\geq 0} c_1(O_{P(N\oplus 1)}(1))^j\cap [P(C\oplus 1)]).$$ So by the projection formula, $$r_*(c(\xi)\cap [P(C\oplus 1)]) = c(N\oplus 1)\cap r_*(\sum_{j\geq 0} c_1(O_{P(N\oplus 1)}(1))^j\cap [P(C\oplus 1)]).$$

Finally, $$r_*(\sum_{j\geq 0} c_1(O_{P(N\oplus 1)}(1))^j\cap [P(C\oplus 1)]) = s(Z,P(E))$$ by definition of the Segre class, and $c(N\oplus 1) = c(N)$, so we have $$r_*(c(\xi)\cap [P(C\oplus 1)]) = c(N)\cap s(Z,P(E)).$$

Since $X$ and thus $P(E)$ are varieties, $P(C\oplus 1)$ is of pure dimension $\dim(P(E))$ \cite[B.6.6]{fulton1}. So as $\xi$ has rank $g$, it is clear that $c(N)\cap s(Z,P(E))$ has no terms of dimension $< \dim P(E) - g$. Thus the class $c(G)\cap i_* s(Z,P(E))\in A_*(P(E))$ has no terms of codimension $> g$. Therefore, via the isomorphism $$A_*(P(E))\cong A_*(X)[t]/(t^e + c_1(E)t^{e - 1} + \ldots + c_e(E)),$$ the class $c(G)\cap i_* s(Z,P(E))$ may be represented by a polynomial in $t$ of degree $\leq g$.
\end{proof}

Note that the term of codimension $g$ of $c(G)\cap i_* s(Z, P(E))$ is actually $i_* P(E)\cdot_G P(E)$, the pushforward of the intersection product of $P(E)$ with itself, where here one $P(E)$ is the image by $s$ of $P(E)$ and the other the image of $P(E)$ by the zero section embedding $s_0: P(E)\hookrightarrow G$ \cite[Proposition 6.1 (a)]{fulton1}. The intersection product respects rational equivalence, and so since the image of $P(E)$ by $s$ is rationally equivalent to the image of $P(E)$ by $s_0$, so we may assume that both $P(E)$ are in fact the image by $s_0$ of $P(E)$, showing that the codimension $g$ term of $c(G)\cap i_* s(Z, P(E))$ must be the codimension $g$ term of $c(G)\cap i_*s(P(E), P(E))$, which is $c_g(G)\cap [P(E)]$. This takes care of part (b) of Proposition 1.3. % this invariance under rational equivalence I believe is a consequence of Fulton's proof of the well-defined-ness of the associated gysin maps with rational equivalence.

Suppose now that we have $\hat{G}, \hat{s},$ and $\hat{Z}$ also defined as in Theorem 1.2. The join operation allows us to express the class $\hat{i}_* s(\hat{Z},P(F))$ in terms of $i_*s(Z,P(E))$. In particular we have the following \cite[Corollary 4.5]{aluffi1}.

\begin{lemma}
If $g < e$, $$\hat{i}_* s(\hat{Z},P(F)) = s(\hat{G})\cap ((c(G)\cap i_*s(Z,P(E)))\vee P(K)),$$ where here $s(\hat{G}) = c(\hat{G})^{-1}$ is the Segre class of $\hat{G}$.
\end{lemma}

So if we assume $g < e$, then the other lemmas show that the minimal degree polynomials in $A_*(X)[t]$ representing $c(G)\cap i_*s(Z,P(E))$ and $c(G)\cap [P(E)]$ remain unchanged after taking the join with $P(K)$, and furthermore that $(c(G)\cap [P(E)])\vee P(K) = c(\hat{G})\cap [P(F)]$. By properties of the intersection product on $A_*(P(F))$ \cite[Example 8.1.6]{fulton1}, $$s(\hat{G})\cap ((c(G)\cap i_*s(Z,P(E)))\vee P(K)) = \frac{(c(G)\cap i_*s(Z,P(E)))\vee P(K)}{c(\hat{G})\cap [P(F)]}.$$  Therefore, the right-hand side of Lemma 2.6 is equal to $\zeta_{G,s}(c_1(O_{P(F)}(1))\cap[P(F)])$, proving Theorem 1.2. % for the Chern class interface with the intersection product, see Fulton Example 8.1.6. Take f to be the identity there.

\section{Extension of bundles and sections}

Parts (a), (b) of Proposition 1.4 follow immediately from several results on the extension of reflexive coherent sheaves from open subschemes which we list here. Recall that a reflexive sheaf on a scheme $Y$ is an $\OOs_Y$-module which is isomorphic to its double dual via the canonical evaluation map. Locally free sheaves are reflexive in particular. We have the following.

\begin{lemma}
If $Y$ is a smooth scheme, $j:U\hookrightarrow Y$ an open immersion with $\codim(Y\setminus U)\geq 2$, and $\GGs$ a reflexive coherent sheaf on $U$, then
\begin{itemize}
\item $j_*\GGs$ is coherent and reflexive,
\item $(j_*\GGs)\big|_U \cong \GGs$,
\item if $\FFs_1, \FFs_2$ are any reflexive coherent sheaves on $Y$ with $\FFs_1\big|_U\cong \FFs_2\big|_U$, then $\FFs_1\cong\FFs_2$.
\end{itemize}
\end{lemma}

These facts are established in greater generality by Hassett and Kov\'acs \cite[Proposition 3.6, Corollary 3.7]{hassett_kovacs1}. Thus in our particular situation, the question of whether the bundle $G$ on $P(E)$ extends to a bundle on $P(F)$ is equivalent to whether the sheaf $j_*(\phi\big|_U)^*\GGs$ is locally free, where here $\GGs$ is the locally free sheaf of sections of $G$ and $j$ is the open immersion $U = P(F)\setminus P(K)\hookrightarrow P(F)$.

It seems to be an interesting question whether this is always true in our situation. In the case where $X$ is a point, and so $P(E) = \PP^{e-1}, P(F) = \PP^{f-1}$ are projective spaces, some related results have been established. For instance, Kempf \cite{kempf1} has proven that a bundle $G$ on $\PP^m$, $m\geq 2$, is a sum of line bundles if and only if the cohomology groups $H^1(\PP^m, \mathscr{H}om (\GGs, \GGs)(-a))$ vanish for all $a>0$ and $G$ extends to a bundle on $\PP^{m+1}$ in the sense that there is a bundle on $\PP^{m+1}$ which restricts to $G$ on $\PP^m\subseteq \PP^{m+1}$, though it turns out the cohomological requirement already implies the extendability of $G$, rendering it a redundant assumption \cite{kumar1}. It is also known that if $G$ extends in this way to $\PP^M$ for every $M > m$, then $G$ has the same total Chern class as a direct sum of line bundles, and if the rank of $G$ is $2$, then in fact $G$ must actually be a sum of line bundles \cite{barth_vdv1}.

Furthermore on $\PP^4$ there is the Horrocks-Mumford bundle, an indecomposable rank $2$ bundle (its total Chern class is $1 + 5H_4 + 10H_4^2\in A_*(\PP^4)$) that fails to extend to a bundle on $\PP^5$ \cite[Theorem 2.7]{hulek1}. % I also believe that the pullback of the tangent bundle on P^{n+1} is an indecomposable bundle on P^n (when n >= 2), which extends to the tangent bundle on P^{n+1}. One has to use the leading term trick of the absolute case to be able to fully determine the Segre zeta function from this bundle, but it can be done, so this could be an explicit example of an indecomposable bundle extending.

In this work we avoid this uncertainty by focusing on bundles which split into direct sums of line bundles as these always extend in our sense. Every subscheme of $P(E)$ may be realized as the zero scheme of a section of such a bundle. Indeed, as we mentioned in Section 1.2, every closed subscheme $Z$ of $P(E)$ is the zero scheme of a section of some bundle on $P(E)$. By \cite[Exercise III.6.8 (b)]{hartshorne1}, this bundle injects into a sum of line bundles, and so we may construct a section of that split bundle whose zero scheme is $Z$.

So if the goal is to associate a relative Segre zeta function to a closed subscheme $Z$ of $P(E)$, the only issue stymieing Theorem 1.2 is whether we can find a sum of line bundles defining $Z$ of sufficiently small rank, specifically of rank $< e$.% Let's say we define a subscheme Z as the zero scheme of a section s of a bundle G. Hartshorne exercise implies that the dual of G is the quotient of a direct sum of line bundles. Since kernels of surjections of bundles are bundles, we can take the dual of the corresponding exact sequence, use the biduality isomorphism and the fact that duality commutes with Hom to conclude that G injects into the dual of that sum of line bundles, which is itself a sum of line bundles. So composing s with this injection gives a section of the sum of line bundles that vanishes precisely when s vanishes. In other words, its zero scheme is also Z. By properties of fiber products and the collective compatibility of the fiber squares defining the two zero schemes, the zero schemes are isomorphic.

Direct sums of line bundles extend to $P(F)$ because of the following well-known fact for which we make note of a brief argument due to lack of appropriate reference.

\begin{lemma}
Let $Y$ be a smooth variety, and $U$ an open subscheme. Suppose $L$ is a line bundle on $U$. Then there exists a line bundle $\hat{L}$ on $Y$ with $\hat{L}\big|_U \cong L$.
\end{lemma}
\begin{proof}
Due to the smoothness of $Y$ there is a bijective correspondence between Weil and Cartier divisors on $Y$ up to linear equivalence, and likewise for $U$. A line bundle $L$ on $U$ corresponds to a Cartier divisor, which then in turn corresponds to some Weil divisor $D := \sum_{i = 1}^d a_i [V_i]$ in $A_{\dim(U) - 1}(U)$. Here the $V_i$ are codimension one subvarieties of $U$, and the $a_i\in \ZZ$. Form the Weil divisor $\overline{D} := \sum_{i = 1}^d a_i [\overline{V_i}]$ in $A_{\dim(Y) - 1}(Y)$. Then $j^*\overline{D} = D$. Furthermore $\overline{D}$ corresponds to a Cartier divisor and thus a line bundle $\hat{L}$ on $Y$ which restricts to $L$.
\end{proof}

Therefore in our original situation, for an invertible sheaf $\LLs$ on $P(E)$, $j_* (\phi\big|_U)^*\LLs$ is also invertible by Lemma 3.1. The functor $j_*(\phi\big|_U)^*$ is additive so if we have a direct sum of invertible sheaves on $P(E)$ we may simply apply it to get a sum of invertible sheaves on $P(F)$.

A section $s$ of a bundle $G$ on $P(E)$ is specified by giving a map $\mathscr{O}_{P(E)}\rightarrow \mathscr{G}$. Assuming $j_* (\phi\big|_U)^* \GGs =: \hat{\GGs}$ is locally free, to get a section of the corresponding bundle $\hat{G}$, we apply $j_* (\phi\big|_U)^*$ to $\mathscr{O}_{P(E)}\rightarrow \mathscr{G}$ and note $$j_* (\phi\big|_U)^* \OOs_{P(E)} \cong \OOs_{P(F)}$$ by Lemma 3.1. This section will be compatible with $s$ in the sense of property (ii) described in the introduction. Note that the condition that the rank of $E$ is $\geq 2$ in Proposition 1.4 is what allows us to satisfy the codimension constraint of Lemma 3.1 to get the desired uniqueness property.

This finishes the proof of Proposition 1.4.

\section{The case of subbundles}

\subsection{When $F = E\oplus L$}

A special situation worth describing is when the surjection of bundles $\phi: F\rightarrow E$ in Theorem 1.2 splits by way of $F$ being a sum of $E$ with a line bundle $L$ on $X$, $F = E\oplus L$. That is, if $$l:E\hookrightarrow E\oplus L = F$$ is the inclusion, $\phi\circ l = \id_E$. In this case Theorem 1.2 admits a lower-tech proof.

That $F = E\oplus L$ implies there is a closed embedding $l: P(E)\hookrightarrow P(F)$ for which $l^*O_{P(F)}(1) = O_{P(E)}(1)$ \cite[B.5.1]{fulton1}. Instead of the join operation defined in Section 2 we may use the Gysin pullback map $$l^*: A_*(P(F))\rightarrow A_*(P(E))$$ induced by this embedding to relate classes in $A_*(P(F)), A_*(P(E))$ \cite[6.2]{fulton1}. The embedding $l$ realizes $Z$ as an effective Cartier divisor in $\hat{Z}$ and we will abuse notation using $l^*$ to also denote the Gysin map $A_*(\hat{Z})\rightarrow A_*(Z)$. % consider the absolute case for instance: there Z is cut out by one equation (the image of the last coordinate in the localized quotient rings) which is never a zero divisor. The same argument should work for the general case, once we work over a trivializing open subset of $X$.

Note $l(P(E))$ is contained in the complement $U$ of the projectivization of the kernel of the vector bundle map $\phi$, $P(0\oplus L)$. Thus the stipulation that the vector bundle surjection splits in the above sense means that $$\phi\big|_U \circ l = \id_{P(E)}.$$ Here we abuse notation again, using $\phi$ to also refer to the induced rational map $\phi: P(F)\dashrightarrow P(E)$. 

Therefore if $\hat{G}$ is a bundle on $P(F)$ compatible with a rank $g$ bundle $G$ on $P(E)$, we have $l^*\hat{G} \cong G$. Further, if $\hat{s}$ is a section of $\hat{G}$ compatible with a section $s$ of $G$, and if $i:Z\hookrightarrow P(E),$ $\hat{i}: \hat{Z}\hookrightarrow P(F)$ are the zero schemes of $s, \hat{s}$, then we see $l^{-1}(\hat{Z}) = Z$.

Since $L$ is a line bundle, $P(E)$ is an effective Cartier divisor in $P(F)$ via $l$. The intersection $P(E)\cap \hat{Z} = Z$ is transverse in the sense that $P(E)$ does not contain the support of any irreducible component of the normal cone of $\hat{Z}$ in $P(F)$ \cite[B.5.3]{fulton1}, and so we may apply the following property of Segre classes, see \cite[Lemma 4.1]{aluffi3}:

\begin{lemma}
Let $Z\hookrightarrow W$ be schemes, $D$ a Cartier divisor on $W$ such that $D$ does not contain the support of any irreducible component of the normal cone of $Z$ in $W$, $C_Z W$. Then $$s(Z\cap D, D) = D\cdot s(Z, W).$$
\end{lemma}

With our Gysin map notation, Lemma 4.1 implies $l^*s(\hat{Z},P(F)) = s(Z,P(E))$. Gysin maps are compatible with proper pushfoward, see \cite[Theorem 6.2 (a)]{fulton1}, and so this applied to the fiber diagram

\begin{center}
\begin{tikzcd}
Z = \hat{Z}\cap P(E) \arrow[r,hookrightarrow, "l"] \arrow[d, hookrightarrow, "i"] & \hat{Z} \arrow[d, hookrightarrow, "\hat{i}"]\\
P(E) \arrow[r, hookrightarrow, "l"] & P(F)
\end{tikzcd}
\end{center}

in conjunction with Lemma 4.1 yields $$l^* \hat{i}_*s(\hat{Z}, P(F)) = i_* s(Z, P(E)).$$

The numerator of the relative Segre zeta function is $$c(G)\cap i_* s(Z, P(E)).$$ Notice that for any $\alpha\in A_*(P(F))$, $$l^* c(\hat{G})\cap \alpha = c(G)\cap l^*\alpha.$$ By Lemma 2.5, using the presentation for $A_*(P(E))$ again, this numerator may be represented by a polynomial of degree $\leq g$ in $t$. Furthermore, we have $$l^* (c(\hat{G})\cap \hat{i}_* s(\hat{Z}, P(F))) = c(G)\cap i_* s(Z, P(E)),$$ and additionally $$l^*c_1(O_{P(F)}(1))\cap [P(F)] = c_1(O_{P(E)}(1))\cap [P(E)].$$ % note l^* is indeed functorial with respect to Chern classes. Fulton treats this more generally, in terms of l.c.i. morphisms, which are closed embeddings followed by smooth morphisms. See Proposition 6.3 of the text.

Thus if we assume that $g < e$ as in Theorem 1.2, the same minimal degree polynomial in $t$ that represents $$c(\hat{G})\cap \hat{i}_* s(\hat{Z}, P(F))$$ also represents $c(G)\cap i_* s(Z, P(E))$. Similarly $c(\hat{G})\cap [P(F)]$, $c(G)\cap [P(E)]$ both have no terms of codimension exceeding $g$, and so as $$l^* c(\hat{G})\cap [P(F)] = c(G)\cap [P(E)],$$ both classes may also be represented by the same minimal degree polynomial in $t$.

Thus the rational expressions $\zeta_{G,s}(t)$ and $\zeta_{\hat{G},\hat{s}}(t)$ are identical, and therefore $$\zeta_{G,s}(c(O_{P(F)}(1))\cap[P(F)]) = \hat{i}_* s(\hat{Z}, P(F)),$$ proving Theorem 1.2.

\subsection{Reduction of the ambient space dimension}

Suppose again we have a rank $g$ bundle $G$ on $P(E)$, a section $s$ of $G$ with zero scheme $i:Z\hookrightarrow P(E)$, and suppose we have a rank $e^\prime := e - 1$ subbundle $E^\prime\hookrightarrow E$ so that $P(E^\prime)$ meets $Z$ transversally in the sense of Lemma 4.1. This time we are not assuming there is a surjection $E\rightarrow E^\prime.$

Let $l: P(E^\prime)\hookrightarrow P(E)$ be the induced closed embedding  and let $i^\prime: Z^\prime\hookrightarrow P(E^\prime)$ be the zero scheme of the section $l^*s$. Then as before, Lemma 4.1 shows $$l^*i_*s(Z,P(E)) = i^\prime_* s(Z^\prime, P(E^\prime)).$$ So with the assumption $g < e^\prime$,  $$\zeta_{G,s}(t) = \zeta_{l^*G, l^*s}(t)$$ completeing the proof of Proposition 1.5.

The use of this result is as a means of simplifying the projective bundle one needs to work within in order to compute $\zeta_{G,s}(t)$. If one has a filtration of subbundles of $E$ each of rank one less than the next but still exceeding $g$, Proposition 1.5 may be applied repeatedly to further reduce the projective bundle dimension. This is an analog of intersecting with hyperplanes in the absolute case, where the ability to reduce the ambient space dimension can have substantial practical benefits \cite[Example 5.11]{aluffi1}.

\section{Examples}

\subsection{Subschemes of projective space}

The result of Theorem 1.2 generalizes the notion of the Segre zeta function of a homogeneous ideal introduced in Aluffi \cite{aluffi1}. Using the notation of Section 1.2, we recover this special case when $X = \{\mathrm{pt}\}$, and $E = O^{\oplus (n+1)}_X$, so $P(E) = \PP^n$.

A homogeneous ideal of the homogeneous coordinate ring of $\PP^n$, say $I\subseteq k[x_0,\ldots,x_n]$, can be written as $I = (F_0,\ldots,F_r)$ for some choice of homogeneous generating set $\{F_0,\ldots,F_r\}$ of $I$. Each $F_j$ is a section of the bundle $O_{\PP^n}(d_j)$, where $d_j := \deg(F_j)$.

So as the closed subscheme of $\PP^n$ defined by $I$, $$i: Z := \Proj(k[x_0,\ldots,x_n]/I)\hookrightarrow \PP^n,$$ is the scheme-theoretic intersection of the effective Cartier divisors (hypersurfaces) cut out by the $F_j$, we see that $Z$ is the zero scheme of a section $s$ of the bundle $$G := \bigoplus_{j = 0}^r O_{\PP^n}(d_j).$$ So in this context, the rank constraint of Theorem 1.2 becomes $r+1 < n$.

Furthermore, note that $A_*(X) \cong \ZZ$ and $A_*(P(E))\cong \ZZ[t]/(t^{N+1})$. The total Chern class of $G$ is then $c(G)\cap [\PP^n] = \prod_{j = 0}^r(1 + d_j H_n)$ where $$H_n = c_1(O_{\PP^n}(1))\cap [\PP^n]\in A_*(\PP^n),$$ the class of a hyperplane, and the Segre zeta function has the form $$\zeta_{G,s}(t) = \frac{P(t)}{\prod_{j = 0}^r (1 + d_j t)},$$ which coincides with the Segre zeta function $\zeta_I(t)$ defined by Aluffi \cite{aluffi1} and recovers its rationality.

More is known about the numerator of this rational function in general, as described in \cite{aluffi1}. In particular since we may obtain $c(G)\cap i_*s(Z,\PP^n)$ as a pushforward of the total Chern class of the universal quotient bundle of $P(G\oplus 1)$, see the proof of Lemma 2.5, and since $G$ is globally generated, it follows this class $c(G)\cap i_*s(Z,\PP^n)$ is \emph{nonnegative} \cite[Example 12.1.7]{fulton1}. So the coefficients of $P(t)$ are nonnegative. Also, by Proposition 1.3, the highest codimension term of $c(G)\cap i_*s(Z,\PP^n)$ is $$c_g(G)\cap [\PP^n] = d_0\cdots d_r H_n^{r+1},$$ so the leading term of $P(t)$ is $d_0\cdots d_r t^{r+1}$.

For a larger projective space, given by $F = O_X^{\oplus(N+1)}$, $P(F) = \PP^N$, $N\geq n$, we can take $\hat{G} = \bigoplus_{j = 0}^r O_{\PP^N}(d_j)$ which is compatible with $G$, and use the obvious compatible section. If $\hat{i}: \hat{Z}\hookrightarrow \PP^N$ is the zero scheme of this section, and if $r + 1 < n$ then by Theorem 1.2, $$\hat{i}_* s(\hat{Z},\PP^N) = \zeta_{G,s}(H_N) = \zeta_{I}(H_N),$$ where now $H_N\in A_*(\PP^N)$ is the hyperplane class.

If we express $\zeta_{G,s}(t)$ as a formal power series, $\sum_{j \geq 0} a_j t^j \in \ZZ[[t]]$, it is then clear that we truncate everything but the first $N+1$ terms of this infinite series to obtain an integral polynomial in $H_N$ equal to $\hat{i}_* s(\hat{Z},\PP^N)$.

\subsection{Subschemes of products of projective spaces}

We may form a Segre zeta function for products of projective spaces and derive its properties as a convenient consequence of the relative Segre zeta function of Theorem 1.2. This formulation of the Segre zeta function has appeared in recent work where it is used to give a formula for the Chern-Schwartz-MacPherson class of a closed subscheme of a smooth projective variety \cite{aluffi2}. The alternate proof of Theorem 1.2 given in Section 4 may be used to derive the relative Segre zeta function we will need here as well as the absolute Segre zeta function of the previous example.

Consider the product $\PP^n\times \PP^m$. Here we may realize this scheme simultaneously as both the projective bundle $P(O_{\PP^n}^{\oplus (m+1)})$ and $P(O_{\PP^m}^{\oplus (n+1)})$. Let $p_n, q_m$ denote the projection maps to the first and second component of the product, respectively.

Fixing coordinates, $\PP^n(x_0:\ldots : x_n), \PP^m(y_0:\ldots : y_m)$, note that any closed subscheme $i: Z\hookrightarrow \PP^n\times\PP^m$ may be written as the zero scheme of a finite collection of polynomials $F_0,\ldots, F_r$, bihomogeneous in the $x_j$ and the $y_j$. For each $j$, let $(a_j, b_j)$ be the bidegree of $F_j$. More precisely, each $F_j$ corresponds to a section of the line bundle $$O_{\PP^n\times\PP^m}(a_j,b_j) := p_n^* O_{\PP^n}(a_j) \otimes q_m^* O_{\PP^m}(b_j).$$

So $Z$ is the zero scheme of a section $s$ of the bundle $$G := \bigoplus_{j = 0}^r O_{\PP^n\times\PP^m}(a_j,b_j).$$ The pushforward of its Segre class $i_* s(Z,\PP^n\times\PP^m)$ is a class in the Chow ring $$A_*(\PP^n\times\PP^m) = \ZZ[s,t]/(s^{n+1}, t^{m+1}),$$ where this isomorphism identifies $s,t$ with the pullbacks of the hyperplane classes in $\PP^n, \PP^m$ by $p_n, q_m$, respectively.

There are unique polynomials $P(s,t), Q(s,t)$ in $s,t$ representing the classes $c(G)\cap i_* s(Z,\PP^n\times\PP^m)$ and $c(G)\cap [\PP^n\times\PP^m]$ respectively, of degree $< n+1$ in $s$ and degree $< m+1$ in $t$. In particular, by the sum formula for Chern classes, $Q(s,t) = \prod_{j=0}^r (1 + a_j s + b_j t)$, provided $r < n, m$.

\begin{definition}
We define the \emph{Segre zeta function} of $G$ with respect to $s$ in this context to be $$\zeta_{G,s}(s,t) := \frac{P(s,t)}{Q(s,t)} = \frac{P(s,t)}{\prod_{j=0}^r (1 + a_j s + b_j t)} \in \ZZ[[s,t]].$$
\end{definition}

The bundle $G$ is globally generated when the $a_j, b_j\geq 0$, and so by analogous reasoning as in the previous example, the coefficients of $P(s,t)$ are nonnegative. By Proposition 1.3, the highest degree term of the numerator with respect to total degree comes from $c_g(G)\cap [\PP^n\times\PP^m]$ and so is $(a_0s + b_0 t)\cdots (a_r s + b_r t)$.

Denote by $H_n, H_m^\prime\in A_*(\PP^n\times\PP^m)$ the pullbacks of the hyperplane classes from $\PP^n, \PP^m$, by $p_n, q_m$ respectively. Then we recover the Segre zeta function defined in Section 1.2 for both the case where we consider $\PP^n\times\PP^m$ to be the projective bundle $P(O_{\PP^n}^{\oplus {m+1}})$ and $P(O_{\PP^m}^{\oplus {n+1}})$ as $\zeta_{G,s}(H_n, t)$ and $\zeta_{G,s}(s, H^\prime_m)$, respectively.

Suppose $N, M$ are integers with $N\geq n, M\geq m$. Then we have a closed subscheme $\hat{i} : \hat{Z}\hookrightarrow \PP^N(x_0:\ldots: x_N)\times\PP^M(y_0:\ldots : y_M)$ which is the zero scheme of the section $\hat{s}$ of the bundle $\hat{G} = \bigoplus_{j = 0}^r O_{\PP^N\times\PP^M}(a_j,b_j)$ on $\PP^N\times\PP^M$ corresponding to the $F_j$. If $r<n,m$, then all of our considerations establish that $$\hat{i}_* s(\hat{Z},\PP^N\times\PP^M) = \zeta_{G,s}(H_N,H^\prime_M).$$

For convenience, we collect these observations into a proposition.

\begin{proposition}
Let $n,m, r\geq 0$, be integers with $r < n,m$ and let $F_0,\ldots,F_r$ be bihomogeneous polynomials in $x_0,\ldots,x_n$, $y_0,\ldots,y_m$ of bidegrees $(a_0,b_0),\ldots,(a_r,b_r)$. Take $$G := \bigoplus_{j = 0}^r O_{\PP^n\times\PP^m}(a_j,b_j),$$ and let $s$ denote the section of $G$ determined by the $F_j$. Then there is a formal power series of the form $$\zeta_{G,s}(s,t) = \frac{P(s,t)}{\prod_{j=0}^r (1 + a_j s + b_j t)} \in \ZZ[[s,t]],$$ such that:
\begin{enumerate}
\item For any $N\geq n, M\geq m$, if $i_{N,M}: Z_{N,M}\hookrightarrow \PP^N\times\PP^M$ is the zero scheme of the section of $\bigoplus_{j = 0}^r O_{\PP^N\times\PP^M}(a_j,b_j)$ determined by the $F_j$, then $$(i_{N,M})_* s(Z_{N,M},\PP^N\times\PP^M) = \zeta_{G,s}(H_N,H^\prime_M)\in A_*(\PP^N\times\PP^M),$$ where $H_N, H^\prime_M$ the pullbacks of the hyperplane classes from $\PP^N, \PP^M$, respectively.
\item The coefficients of the polynomial $P(s,t)\in \ZZ[s,t]$ are nonnegative.
\item $P(s,t)$ is of degree $\leq n$ in $s$, of degree $\leq m$ in $t$, and its term of highest total degree is $(a_0s + b_0 t)\cdots (a_r s + b_r t)$. Its lowest total degree term is of total degree $\codim(Z)$, and represents the top dimension part of the pushforward of the fundamental class $[Z]$ in $A_*(\PP^n\times\PP^m)$.
\end{enumerate} 
\end{proposition}

\end{document}